\documentclass[reqno]{amsart}%
\usepackage{amstext}
\usepackage{amsfonts}
\usepackage{amsmath}
\usepackage{amssymb}
\usepackage{hyperref}
\usepackage{graphicx}
\usepackage{cite}%
\providecommand{\U}[1]{\protect\rule{.1in}{.1in}}
\numberwithin{equation}{section}
\newtheorem{theorem}{Theorem}[section]

\newtheorem{lemma}[theorem]{Lemma}

\begin{document}
\title[Lower bounds for the 3D Navier-Stokes equations ]{Lower bounds on the blow-up rate of the 3D Navier-Stokes equations in
H\symbol{94}\{5/2\}}
\author{Abdelhafid Younsi}
\address{Department of Mathematics and Computer Science, University of Djelfa , Algeria.}
\email{younsihafid@gmail.com}

\begin{abstract}
Under assumption that $T^{\ast}$ is the maximal time of existence of smooth
solution of the 3D Navier-Stokes equations in the Sobolev space $H^{s}$, we
establish lower bounds for the blow-up rate of the type$\ \left(  T^{\ast
}-t\right)  ^{-\varphi\left(  n\right)  }$, where $n$ is a natural number
independent of $s$ and $\varphi$ is a linear function. Using this new type in
the 3D Navier-Stokes equations in the $H^{5/2}$, both on the whole space and
in the periodic case, we give an answer to a question left open by James et al
(2012, J. Math. Phys.). We also prove optimal lower bounds for the blow-up
rate in $\dot{H}^{3/2}$ and in $H^{1}$.

\end{abstract}
\maketitle

\section{Introduction}

We consider, in this paper, the 3D incompressible Navier-Stokes equations%
\begin{equation}%
\begin{array}
[c]{c}%
\dfrac{\partial u}{\partial t}+u.\nabla u=-\nabla p+\nu\triangle u,\text{ in
}\Omega\times\left(  0,\infty\right) \\
\text{div }u=0,\text{ in }\Omega\times\left(  0,T\right)  \text{ and }u\left(
x,0\right)  =u_{0},\text{ in }\Omega\text{,}%
\end{array}
\label{1}%
\end{equation}
where $u=u\left(  x,t\right)  \ $is the velocity vector field, $p$ is the
pressure and $\nu$ is the viscosity of the fluid. The domain $\Omega$ may have
periodic boundary conditions or $\Omega=%
%TCIMACRO{\U{211d} }%
%BeginExpansion
\mathbb{R}
%EndExpansion
^{3}$.

For small data $\left\Vert \nabla u\right\Vert _{L^{2}}\leq c\left(
\nu\right)  $ the global existence of strong solutions for the 3D
Navier-Stokes equations it is well known, see Constantin \cite[Theorem 9.3 P
80]{2}. But for the 3D Navier-Stokes equations with large data, we don't have
a result of global existence. Under the assumption that the solution of the
three-dimensional Navier-Stokes equations becomes irregular at finite time
$T^{\ast}$ Leray 1934 \cite[P 224]{4} proved that there exists a constant
$c\left(  \nu\right)  >0$ such that
\begin{equation}
\left\Vert \nabla u\left(  .,t\right)  \right\Vert _{\dot{H}^{1}\left(
%TCIMACRO{\U{211d} }%
%BeginExpansion
\mathbb{R}
%EndExpansion
^{3}\right)  }^{4}\geq\frac{c\left(  \nu\right)  }{\left(  T^{\ast}-t\right)
}. \label{2}%
\end{equation}
In 2010, Benameur \cite[Theorem 1.3.]{1} showed in the whole space
\begin{equation}
\left\Vert u\left(  .,t\right)  \right\Vert _{\dot{H}^{s}\left(
%TCIMACRO{\U{211d} }%
%BeginExpansion
\mathbb{R}
%EndExpansion
^{3}\right)  }\geq c\left(  s\right)  \frac{\left\Vert u\left(  .,t\right)
\right\Vert _{L^{2}\left(
%TCIMACRO{\U{211d} }%
%BeginExpansion
\mathbb{R}
%EndExpansion
^{3}\right)  }^{\frac{3-2s}{3}}}{\left(  T^{\ast}-t\right)  ^{\frac{s}{3}}%
}\text{ with }s>\frac{5}{2}. \label{3}%
\end{equation}
The result above was improved by Robinson, Sadowski, and Silva in \cite{5} to
\begin{equation}
\left\Vert u\left(  .,t\right)  \right\Vert _{\dot{H}^{s}\left(
\Omega\right)  }\geq c\left(  s\right)  \frac{\left\Vert u_{0}\right\Vert
_{L^{2}\left(  \Omega\right)  }^{\frac{5-2s}{5}}}{\left(  T^{\ast}-t\right)
^{\frac{2s}{3}}}\text{ with }\Omega=\left[  0,1\right]  ^{3}\text{ or \ }%
%TCIMACRO{\U{211d} }%
%BeginExpansion
\mathbb{R}
%EndExpansion
^{3}\text{.} \label{4}%
\end{equation}
In the homogeneous Sobolev space $\dot{H}^{5/2}\left(  \mathbb{T}^{3}\right)
$ of real valued periodic functions, Cortissoz, Montero, \& Pinilla 2014
\cite[Theorem 1.3.Page 2]{3} proved lower bounds on the blow up with
logarithmic corrections,
\begin{equation}
\left\Vert u\left(  .,t\right)  \right\Vert _{\dot{H}^{5/2}\left(
\mathbb{T}^{3}\right)  }\geq\frac{c}{\left(  T^{\ast}-t\right)  \left\vert
\lg\left(  T^{\ast}-t\right)  \right\vert }\text{\ with }T^{\ast}-t\neq1.
\label{5}%
\end{equation}
In the case of the $\dot{H}^{3/2}\left(  Q\right)  -$norm, (see also the work
of Robinson et al \cite{5}) it is known that for every $\varepsilon>0$ there
exists a $c_{\frac{3}{2},\varepsilon}>0$ such that
\begin{equation}
\left\Vert u\left(  .,t\right)  \right\Vert _{\dot{H}^{3/2}}^{4+\frac{\gamma
}{3}}\geq\frac{c_{\frac{3}{2},\varepsilon}}{\left(  T^{\ast}-t\right)
^{\frac{1}{2}-\varepsilon}}. \label{6}%
\end{equation}
Recently, Cortissoz et al \cite[Theorem 1.1.Page 2]{3} showed that%
\begin{equation}
\left\Vert u\left(  .,t\right)  \right\Vert _{\dot{H}^{3/2}\left(
\mathbb{T}^{3}\right)  }\geq\frac{c}{\sqrt{\left(  T^{\ast}-t\right)
\left\vert \lg\left(  T^{\ast}-t\right)  \right\vert }}. \label{7}%
\end{equation}
In this paper, we concentrate on the case of estimating lower bounds for the
blow-up rate in $\dot{H}^{5/2}\left(  \mathbb{T}^{3}\right)  $ and $\dot
{H}^{3/2}\left(  Q\right)  $ of a possible blow-up solution to 3D
Navier-Stokes equations. First we prove a new lower bound on blowup solutions
in the $\dot{H}^{5/2}$-norm both on the whole space and in the periodic case.
This result gives a response to the question left open in \cite{5} and
improves previous known lower bounds. Therefore is possible to prove the
correct rate of blow up $(\ref{6})$ in $\dot{H}^{3/2}\left(  Q\right)  $,
which is, in this case, $\left(  T^{\ast}-t\right)  ^{-\frac{1}{2}}$, see
\cite[SecV.A P11)]{5}. Finally, we improve the order $\frac{-1}{4}$ in the
result $(\ref{2})$ to get a rate of the order $-1$. We prove that is possible
to get a rate of blowup of the type $\left(  T^{\ast}-t\right)  ^{-1}$ in
several spaces $\dot{H}^{5/2}$, $\dot{H}^{3/2}\ $and $\dot{H}^{1}$ for $t\leq
T_{\ast}$\ with $T_{\ast}<T^{\ast}$. Those estimates are particularly useful
for obtaining a control of degree $-1$ for the comportment of strong solutions
before the moment of the blow up.

The technique used in this paper are fairly standard based on the properties
of trigonometric functions. This method can be used to estimate lower bounds
on solutions that blowup at some finite time $T^{\ast}>0$ for the strict
positive solution of ordinary differential inequality of the type $y^{\prime
}\leq y^{n}$ with $0<n\leq3$.

\section{Main Result}

Let $Q=\left[  0,2\pi\right]  ^{3}$, we write $%
%TCIMACRO{\U{2124} }%
%BeginExpansion
\mathbb{Z}
%EndExpansion
^{3}=%
%TCIMACRO{\U{2124} }%
%BeginExpansion
\mathbb{Z}
%EndExpansion
^{3}/\left\{  0,0,0\right\}  $, let $\dot{H}^{s}\left(  Q\right)  $ be the
subspace of the Sobolev space $H^{s}$ consisting of divergence-free,
zero-average, periodic real functions,
\begin{equation}
\dot{H}^{s}\left(  Q\right)  =\left\{  u=\sum_{\xi\in\dot{%
%TCIMACRO{\U{2124}}%
%BeginExpansion
\mathbb{Z}%
%EndExpansion
}^{3}}\hat{u}_{\xi}e^{-i\xi.x}:\ \overline{\hat{u}_{\xi}}=\hat{u}_{-\xi}\text{
},\sum_{\xi}\left\vert \xi\right\vert ^{2s}\left\vert \hat{u}_{\xi}\right\vert
^{2}<\infty\ \text{and }\xi.\hat{u}_{\xi}=0\right\}  \label{8}%
\end{equation}
and equip $\dot{H}^{s}\left(  Q\right)  \ $ with the norm%
\begin{equation}
\left\Vert u\right\Vert _{\dot{H}^{s}}^{2}=\left\Vert u\right\Vert _{\dot
{H}^{s}\left(  Q\right)  }^{2}=\sum_{\xi}\left\vert \xi\right\vert
^{2s}\left\vert \hat{u}_{\xi}\right\vert ^{2}. \label{9}%
\end{equation}
On the whole space the corresponding definition of the $\dot{H}^{s}\left(
%TCIMACRO{\U{211d} }%
%BeginExpansion
\mathbb{R}
%EndExpansion
^{3}\right)  $ norm is
\begin{equation}
\left\Vert u\right\Vert _{\dot{H}^{s}\left(
%TCIMACRO{\U{211d} }%
%BeginExpansion
\mathbb{R}
%EndExpansion
^{3}\right)  }^{2}:=\int_{%
%TCIMACRO{\U{211d} }%
%BeginExpansion
\mathbb{R}
%EndExpansion
^{3}}\left\vert \xi\right\vert ^{2s}\left\vert \hat{u}\left(  \xi\right)
\right\vert ^{2}d\xi<\infty, \label{10}%
\end{equation}
where $F[u](\xi)=\hat{u}(\xi)=\int_{%
%TCIMACRO{\U{211d} }%
%BeginExpansion
\mathbb{R}
%EndExpansion
^{n}}e^{-i\xi..x}\left\vert u\left(  x\right)  \right\vert dx$ is the Fourier
transform of $u$, for more details see \cite{5}. We prove our estimate in the
periodic case, but it also holds in the full space. Throughout the paper,
$c_{i},i\in%
%TCIMACRO{\U{2115} }%
%BeginExpansion
\mathbb{N}
%EndExpansion
$ , denotes a positive constant. We start by proving the following lemma

\begin{lemma}
For any real numbers $T^{\ast}>0$ there exists $m\in%
%TCIMACRO{\U{2115} }%
%BeginExpansion
\mathbb{N}
%EndExpansion
^{\star}$ such that we have
\begin{equation}
\left(  T^{\ast}-t\right)  \leq m^{n}\left(  T^{\ast}-t\right)  ^{n+1}
\label{10a}%
\end{equation}
for all $t\in\left[  0,T^{\ast}-\dfrac{1}{m}\right]  $ and $n\in%
%TCIMACRO{\U{2115} }%
%BeginExpansion
\mathbb{N}
%EndExpansion
^{\star}$.
\end{lemma}

\begin{proof}
We will prove $(\ref{10a})$ by induction. We show that the statement
$(\ref{10a})$ holds for $n=1$. Since $%
%TCIMACRO{\U{211d} }%
%BeginExpansion
\mathbb{R}
%EndExpansion
_{\ast}^{+}$ is archimedean then for any finite real number $T^{\ast}$
strictly positive there exists $m\in%
%TCIMACRO{\U{2115} }%
%BeginExpansion
\mathbb{N}
%EndExpansion
^{\bigstar}$ such that $T^{\ast}\geq\frac{1}{m}$. For any real $t$ such that
$0$ $\leq$ $t\leq T^{\ast}-\frac{1}{m}$, the following inequality holds%
\begin{equation}
\frac{1}{m}\leq\left(  T^{\ast}-t\right)  . \label{11a}%
\end{equation}
Since $\left(  T^{\ast}-t\right)  \geq0$, multiplying $(\ref{11a})$ by
$\left(  T^{\ast}-t\right)  $ we find
\begin{equation}
\left(  T^{\ast}-t\right)  \leq m\left(  T^{\ast}-t\right)  ^{2}, \label{13a}%
\end{equation}
thus $(\ref{10a})$ is true for $n=1$. Let $k\in%
%TCIMACRO{\U{2115} }%
%BeginExpansion
\mathbb{N}
%EndExpansion
^{\star}$ be given and suppose $(\ref{10a})$ is true for $n=k$. Then
\begin{equation}
\left(  T^{\ast}-t\right)  \leq m^{k}\left(  T^{\ast}-t\right)  ^{k+1}.
\label{14a}%
\end{equation}
Multiplying the induction hypothesis $(\ref{14a})$ by $\left(  T^{\ast
}-t\right)  $ we find%
\begin{equation}
\left(  T^{\ast}-t\right)  ^{2}\leq m^{k}\left(  T^{\ast}-t\right)  ^{k+2}.
\label{15a}%
\end{equation}
Using inequality $(\ref{13a})$, we can rewrite $(\ref{15a})$ as
\begin{equation}
\left(  T^{\ast}-t\right)  \leq m^{k+1}\left(  T^{\ast}-t\right)  ^{k+2}
\label{16a}%
\end{equation}
Thus, $(\ref{10a})$ holds for $n=k+1$, and the proof of the induction step is
complete. Wich prove that $(\ref{10a})$ is true for all $n\in%
%TCIMACRO{\U{2115} }%
%BeginExpansion
\mathbb{N}
%EndExpansion
^{\star}$.
\end{proof}

Note that the lemma above holds if we replace $m$ with $r\in%
%TCIMACRO{\U{211d} }%
%BeginExpansion
\mathbb{R}
%EndExpansion
_{+}^{\star}$ and $r\geq1$.

\begin{theorem}
Let $u\left(  .,t\right)  \in\dot{H}^{5/2}\left(  Q\right)  $ be a smooth
Leray-Hopf solution of the 3D Navier-Stokes equations $(\ref{1})$ with non
zero $\left\Vert u_{0}\right\Vert _{\dot{H}^{5/2}}$ and\ with maximal interval
of existence $(0,T^{\ast})$, $T^{\ast}<\infty$. Then there exists a positive
time $T_{\ast}$ $<T^{\ast}$and a positive constant $\eta_{1}>0$ such that%
\begin{equation}
\left\Vert u\left(  .,t\right)  \right\Vert _{\dot{H}^{5/2}}\geq\frac{\eta
_{1}}{\left(  T^{\ast}-t\right)  ^{\frac{n+1}{2}}}\text{ for }t\leq T_{\ast}.
\label{11}%
\end{equation}

\end{theorem}

\begin{proof}
We start our proof from the fourth differential inequality \cite[(V.A
P11)]{5}
\begin{equation}
\frac{d}{dt}\left\Vert u\left(  .,t\right)  \right\Vert _{\dot{H}^{5/2}}%
^{2}\leq c_{1}\left\Vert u\left(  .,t\right)  \right\Vert _{L^{2}}^{4\xi
}\left\Vert u\left(  .,t\right)  \right\Vert _{\dot{H}^{5/2}}^{3+\xi
}\ \text{with }\xi=\frac{\epsilon}{5\left(  4-\epsilon\right)  }, \label{12}%
\end{equation}
for $0\leq\epsilon\leq1\ $yields$\ 3\leq3+\xi\leq\frac{46}{15}$. Since
$\left\Vert u\left(  .,t\right)  \right\Vert _{L^{2}}$ is bounded it follows
that%
\begin{equation}
\frac{d}{dt}\left\Vert u\left(  .,t\right)  \right\Vert _{\dot{H}^{5/2}}%
^{2}\leq c_{2}\left\Vert u\left(  .,t\right)  \right\Vert _{\dot{H}^{5/2}%
}^{3+\xi}. \label{13}%
\end{equation}
Setting $z\left(  t\right)  =\left\Vert u\left(  .,t\right)  \right\Vert
_{\dot{H}^{5/2}}^{2}$ in $(\ref{13})$, we can obtain%
\begin{equation}
\frac{d\left(  z\left(  t\right)  +1\right)  }{dt}\leq c_{2}z^{\frac{3+\xi}%
{2}}\left(  t\right)  . \label{14}%
\end{equation}
Multiplying $(\ref{14})$ by $\cos\left(  \frac{1}{z\left(  t\right)
+1}\right)  $, and using $\cos\left(  1\right)  \leq\cos\left(  \frac
{1}{z\left(  t\right)  +1}\right)  \leq1$, we get
\begin{equation}
\cos\left(  \frac{1}{z\left(  t\right)  +1}\right)  \frac{d\left(  z\left(
t\right)  +1\right)  }{dt}\leq c_{2}z^{\frac{3+\xi}{2}}\left(  t\right)  .
\label{15}%
\end{equation}
Dividing $(\ref{15})$ by $\left(  z+1\right)  ^{2}$, and using $\frac{z^{n}%
}{\left(  z+1\right)  ^{2}}\leq1$ for $0\leq n\leq2$ , we obtain$\ $%
\begin{equation}
\frac{\left(  z\left(  t\right)  +1\right)  \prime}{\left(  z\left(  t\right)
+1\right)  ^{2}}\cos\left(  \frac{1}{z\left(  t\right)  +1}\right)  \leq
c_{2}. \label{16}%
\end{equation}
Integrating the differential inequality $(\ref{16})$ from time $t$ to blow-up
time $T^{\ast}$ and using the fact that $\lim\limits_{t\rightarrow T^{\ast}%
}\left\Vert u\left(  .,t\right)  \right\Vert _{\dot{H}^{5/2}}^{2}=\infty$,
yields
\begin{equation}
\sin\left(  \frac{1}{z\left(  t\right)  +1}\right)  \leq c_{2}\left(  T^{\ast
}-t\right)  . \label{17}%
\end{equation}
Using $(\ref{10a})$ in $(\ref{17})$, we obtain the following estimate
\begin{equation}
\sin\left(  \frac{1}{z\left(  t\right)  +1}\right)  \leq c_{2}m^{n}\left(
T^{\ast}-t\right)  ^{n+1}, \label{18}%
\end{equation}
with $m\in%
%TCIMACRO{\U{2115} }%
%BeginExpansion
\mathbb{N}
%EndExpansion
^{\bigstar}$ for all $t\leq T^{\ast}-\frac{1}{m}=T_{\ast}$.
Multiplying$\ (\ref{18})$ by $\left\Vert u\left(  .,t\right)  \right\Vert
_{\dot{H}^{5/2}}^{2}$, we get
\begin{equation}
\left\Vert u\left(  .,t\right)  \right\Vert _{\dot{H}^{5/2}}^{2}\sin\left(
\frac{1}{\left\Vert u\left(  .,t\right)  \right\Vert _{\dot{H}^{5/2}}^{2}%
+1}\right)  \leq c_{2}m^{n}\left(  T^{\ast}-t\right)  ^{n+1}\left\Vert
u\left(  .,t\right)  \right\Vert _{\dot{H}^{5/2}}^{2}. \label{20}%
\end{equation}
Equation $(\ref{13})$ is a differential equation of Bernoulli type
\begin{equation}
y^{\prime}=y^{n}\text{ with }n>1. \label{20a}%
\end{equation}
To get in the left-hand side of $(\ref{20})$ a minimum different to zero, we
must assume that $\left\Vert u\right\Vert _{\dot{H}^{5/2}}\neq0$. For $n>1$,
$y=0$ must be the only solution to $(\ref{20a})$ satisfying $y\left(
t_{0}\right)  =0$. Then yields that $\left\Vert u\right\Vert _{\dot{H}^{5/2}}$
is non-zero, for $\left\Vert u_{0}\right\Vert _{\dot{H}^{5/2}}\neq0$. Thus,
there exist a positive constant $\beta_{1}=\min\limits_{t\geq0}\left\Vert
u\left(  .,t\right)  \right\Vert _{\dot{H}^{5/2}}^{2}$ such that $\left\Vert
u\left(  .,t\right)  \right\Vert _{\dot{H}^{5/2}}^{2}\geq\beta_{1}$ for all
$t\geq0$. Note that for $\theta\geq0$, the function $f\left(  \theta\right)
=\theta\sin\left(  \frac{1}{\theta+1}\right)  $ is an increasing positive
function, this gives
\begin{equation}
\alpha_{1}=\beta_{1}\sin\left(  \frac{1}{\beta_{1}+1}\right)  \leq\left\Vert
u\right\Vert _{\dot{H}^{5/2}}^{2}\sin\left(  \frac{1}{\left\Vert u\right\Vert
_{\dot{H}^{5/2}}^{2}+1}\right)  . \label{21}%
\end{equation}
Using this estimate in $(\ref{20})$ yields the bound
\begin{equation}
\alpha_{1}\leq c_{2}m^{n}\left(  T^{\ast}-t\right)  ^{n+1}\left\Vert u\left(
.,t\right)  \right\Vert _{\dot{H}^{5/2}}^{2}. \label{22}%
\end{equation}
Then we can deduce $(\ref{11})$ with $\eta_{1}=\sqrt{\dfrac{\alpha_{1}}%
{c_{2}m^{n}}}$, which completes the proof.
\end{proof}

Theorem 2.1 is valid when we consider the case of the whole space, i.e., for
solutions $u(t)$ $\in\dot{H}^{5/2}\left(
%TCIMACRO{\U{211d} }%
%BeginExpansion
\mathbb{R}
%EndExpansion
^{3}\right)  $, this because the equation $(\ref{12})$ in the proof valid in
the whole space and for periodic boundary conditions see \cite[SecV.A P11)]{5}
and all the calculations leading to its proof are valid on $%
%TCIMACRO{\U{211d} }%
%BeginExpansion
\mathbb{R}
%EndExpansion
^{3}$ if we change Fourier series by Fourier integrals.

Therefore is possible to prove the improve the rate of blow up $(\ref{6})$ in
$\dot{H}^{3/2}\left(  Q\right)  $ obtained in \cite{5}.

\begin{theorem}
Let $u\left(  .,t\right)  \in\dot{H}^{3/2}\left(  Q\right)  $ be a smooth
Leray-Hopf solution of the 3D Navier-Stokes equations $(\ref{1})$ with non
zero $\left\Vert u_{0}\right\Vert _{\dot{H}^{3/2}}$ and\ with maximal interval
of existence $(0,T^{\ast})$, $T^{\ast}<\infty$. Then there exists a positive
time $T_{\ast}$ $<T^{\ast}$and a constant $\eta_{2}>0$ such that%
\begin{equation}
\left\Vert u\left(  .,t\right)  \right\Vert _{\dot{H}^{3/2}}\geq\frac{\eta
_{2}}{\left(  T^{\ast}-t\right)  ^{\frac{n+1}{4}}}\text{ for }t\leq T_{\ast}.
\label{23}%
\end{equation}

\end{theorem}

\begin{proof}
We start with the following inequality (inequality $\left(  7\right)  $ in
\cite[ P7]{5})
\begin{equation}
\frac{d}{dt}\left\Vert u\left(  .,t\right)  \right\Vert _{\dot{H}^{3/2}}%
^{2}\leq c_{3}\left\Vert u\left(  .,t\right)  \right\Vert _{L^{2}}^{2\gamma
}\left\Vert u\left(  .,t\right)  \right\Vert _{\dot{H}^{3/2}}^{4+\frac{\gamma
}{3}}\ \text{with }\gamma=\frac{2\delta}{\left(  2-\delta\right)  },
\label{24}%
\end{equation}
for $\delta>0\ $small. Since $\left\Vert u\left(  .,t\right)  \right\Vert
_{L^{2}}$ is bounded, we obtain%
\begin{equation}
\frac{d}{dt}\left\Vert u\left(  .,t\right)  \right\Vert _{\dot{H}^{3/2}}%
^{2}\leq c_{4}\left\Vert u\left(  .,t\right)  \right\Vert _{\dot{H}^{3/2}%
}^{4+\frac{\gamma}{3}}\ . \label{25}%
\end{equation}
Setting $y\left(  t\right)  =\left\Vert u\left(  .,t\right)  \right\Vert
_{\dot{H}^{3/2}}^{2}$, the inequality $(\ref{25})$ can be written in the form%
\begin{equation}
\frac{d\left(  y\left(  t\right)  +1\right)  }{dt}\leq c_{4}y^{2+\frac{\gamma
}{6}}\left(  t\right)  . \label{26}%
\end{equation}
Multiplying $(\ref{26})$ by $\sin\left(  \frac{1}{y\left(  t\right)
+1}\right)  $, and using $y\sin\left(  \frac{1}{y\left(  t\right)  +1}\right)
\leq1$, we find
\begin{equation}
\sin\left(  \frac{1}{y\left(  t\right)  +1}\right)  \frac{d\left(  y\left(
t\right)  +1\right)  }{dt}\leq c_{4}y^{1+\frac{\gamma}{6}}\left(  t\right)  .
\label{27}%
\end{equation}
Dividing $(\ref{27})$ by $\left(  y+1\right)  ^{2}$, and using $\frac{y^{n}%
}{\left(  y+1\right)  ^{2}}\leq1$ for $0\leq n\leq2$, we obtain$\ $%
\begin{equation}
\frac{\left(  y\left(  t\right)  +1\right)  \prime}{\left(  y\left(  t\right)
+1\right)  ^{2}}\sin\left(  \frac{1}{y\left(  t\right)  +1}\right)  \leq
c_{4}. \label{28}%
\end{equation}
Integrating the differential inequality $(\ref{28})$ from time $t$ to blow-up
time $T^{\ast}$ and using the fact that $\lim\limits_{t\rightarrow T^{\ast}%
}\left\Vert u\left(  .,t\right)  \right\Vert _{\dot{H}^{3/2}}^{2}=\infty$, we
find that
\begin{equation}
1\leq\cos\left(  \frac{1}{z\left(  t\right)  +1}\right)  +c_{4}\left(
T^{\ast}-t\right)  . \label{29}%
\end{equation}
Using $(\ref{18})$ in $(\ref{29})$, it follows that
\begin{equation}
1-\cos\left(  \frac{1}{z\left(  t\right)  +1}\right)  \leq c_{4}m^{n}\left(
T^{\ast}-t\right)  ^{n+1}\text{ with }m\in%
%TCIMACRO{\U{2115} }%
%BeginExpansion
\mathbb{N}
%EndExpansion
^{\bigstar}, \label{30}%
\end{equation}
for all $t\leq T_{\ast}\ $(see proof of Theorem 2.1). Using the trigonometric
formula $\sin\left(  \theta\right)  ^{2}=\frac{1-\cos\left(  2\theta\right)
}{2}$, we obtain%
\begin{equation}
\sin\left(  \frac{1}{2z\left(  t\right)  +2}\right)  ^{2}\leq c_{4}%
m^{n}\left(  T^{\ast}-t\right)  ^{n+1}. \label{31}%
\end{equation}
Multiplying$\ (\ref{31})$ by $\left\Vert u\left(  .,t\right)  \right\Vert
_{\dot{H}^{3/2}}^{4}$, we get
\begin{equation}
\left\Vert u\left(  .,t\right)  \right\Vert _{\dot{H}^{3/2}}^{4}\sin\left(
\frac{1}{2\left\Vert u\left(  .,t\right)  \right\Vert _{\dot{H}^{3/2}}^{2}%
+2}\right)  ^{2}\leq c_{4}m^{n}\left(  T^{\ast}-t\right)  ^{n+1}\left\Vert
u\left(  .,t\right)  \right\Vert _{\dot{H}^{3/2}}^{4}. \label{32}%
\end{equation}
For $\left\Vert u_{0}\right\Vert _{\dot{H}^{3/2}}\neq0$, there exists a
positive constant $\alpha_{2}$ (see proof of Theorem 2.1) such that%
\begin{equation}
\alpha_{2}=\left(  \beta_{2}\sin\left(  \frac{1}{2\beta_{2}+2}\right)
\right)  ^{2}\leq\left\Vert u\left(  .,t\right)  \right\Vert _{\dot{H}^{3/2}%
}^{4}\sin\left(  \frac{1}{2\left\Vert u\left(  .,t\right)  \right\Vert
_{\dot{H}^{3/2}}^{2}+2}\right)  ^{2}. \label{33}%
\end{equation}
Using this estimate in $(\ref{33})$, we find that
\begin{equation}
\alpha_{2}\leq c_{4}m^{n}\left(  T^{\ast}-t\right)  ^{n+1}\left\Vert u\left(
.,t\right)  \right\Vert _{\dot{H}^{3/2}}^{4} \label{34}%
\end{equation}

Then we can deduce that $\eta_{2}=\left(  \dfrac{\alpha_{2}}{c_{4}m^{n}%
}\right)  ^{\frac{1}{4}}$, which completes the proof.
\end{proof}

We now present another application of this trigonometric method, leading to
improve the lower bound $(\ref{2}).$

\begin{theorem}
Let $u\left(  .,t\right)  \in\dot{H}^{1}\left(  Q\right)  $ be a smooth
Leray-Hopf solution of the 3D Navier-Stokes equations $(\ref{1})$ with non
zero $\left\Vert u_{0}\right\Vert _{\dot{H}^{1}}$ and\ with maximal interval
of existence $(0,T^{\ast})$, $T^{\ast}<\infty$. Then there exists a positive
time $T_{\ast}$ $<T^{\ast}$and a constant $\eta_{3}>0$ such that%
\begin{equation}
\left\Vert u\left(  .,t\right)  \right\Vert _{\dot{H}^{1}}\geq\eta_{3}\left(
T^{\ast}-t\right)  ^{-\frac{n+1}{4}}\text{ for }t\leq T_{\ast}. \label{35}%
\end{equation}

\end{theorem}

\begin{proof}
We consider the Navier--Stokes equations $(\ref{1})$ in periodic domain.
Multiplying $(\ref{1})$ by $\triangle u$ and integrate, we obtain
\begin{equation}
\frac{1}{2}\frac{d}{dt}\left\Vert \nabla u\left(  .,t\right)  \right\Vert
_{L^{2}}^{2}+\nu\Vert\triangle u\Vert_{L^{2}}^{2}=\left(  \left(  u.\nabla
u\right)  .\triangle u\right)  .\label{36}%
\end{equation}
Using the Holder inequality and the Sobolev theorem, we get
\begin{equation}
\left\vert \left(  \left(  u.\nabla u\right)  .\triangle u\right)  \right\vert
\leq c_{5}\left\Vert \nabla u\right\Vert _{L^{2}}^{\frac{3}{2}}\left\Vert
\triangle u\right\Vert _{L^{2}}^{\frac{3}{2}},\label{37}%
\end{equation}
see \cite[P. 79 (2.22)]{4}. Combining $(\ref{36})$ and $(\ref{37})$, we obtain%
\begin{equation}
\frac{d}{dt}\left\Vert \nabla u\left(  .,t\right)  \right\Vert _{L^{2}}%
^{2}+\nu\Vert\triangle u\Vert_{L^{2}}^{2}\leq c_{5}\left\Vert \nabla
u\right\Vert _{L^{2}}^{\frac{3}{2}}\left\Vert \triangle u\right\Vert _{L^{2}%
}^{\frac{3}{2}}.\label{38}%
\end{equation}
However, an application of Young's inequality to the he right-hand side of
$(\ref{38})$ yields%
\begin{equation}
\frac{d}{dt}\left\Vert \nabla u\left(  .,t\right)  \right\Vert _{L^{2}}%
^{2}+\nu\Vert\triangle u\Vert_{L^{2}}^{2}\leq c_{6}\left\Vert \nabla
u\right\Vert _{L^{2}}^{6}+\frac{\nu}{2}\left\Vert \triangle u\right\Vert
_{L^{2}}^{2}.\label{39}%
\end{equation}
We obtain
\begin{equation}
\frac{d}{dt}\left\Vert \nabla u\left(  .,t\right)  \right\Vert _{L^{2}}%
^{2}+\nu\Vert\triangle u\Vert_{L^{2}}^{2}\leq c_{6}\left\Vert \nabla
u\right\Vert _{L^{2}}^{6}.\label{40}%
\end{equation}
If we drop the $\nu\Vert\triangle u\Vert_{L^{2}}^{2}$ term in $(\ref{40})$
then we have%
\begin{equation}
\frac{d}{dt}\left\Vert \nabla u\left(  .,t\right)  \right\Vert _{L^{2}}%
^{2}\leq c_{6}\left\Vert \nabla u\right\Vert _{L^{2}}^{6}.\label{41}%
\end{equation}
Setting $y\left(  t\right)  =\left\Vert \nabla u\left(  .,t\right)
\right\Vert _{L^{2}}^{2}$ in $(\ref{41})$, this gives%
\begin{equation}
\frac{d}{dt}y\leq c_{6}y^{3}.\label{42}%
\end{equation}
Multiplying $(\ref{42})$ by $\sin\left(  \frac{1}{y\left(  t\right)
+1}\right)  $, and using $y\sin\left(  \frac{1}{y\left(  t\right)  +1}\right)
\leq1$, we get
\begin{equation}
\sin\left(  \frac{1}{y\left(  t\right)  +1}\right)  \frac{d\left(  y\left(
t\right)  +1\right)  }{dt}\leq c_{6}y^{2}\left(  t\right)  .\label{43}%
\end{equation}
Dividing $(\ref{43})$ by $\left(  y+1\right)  ^{2}$, and using $\frac{y^{n}%
}{\left(  y+1\right)  ^{2}}\leq1$ for $0\leq n\leq2$, we obtain$\ $%
\begin{equation}
\frac{\left(  y\left(  t\right)  +1\right)  \prime}{\left(  y\left(  t\right)
+1\right)  ^{2}}\sin\left(  \frac{1}{y\left(  t\right)  +1}\right)  \leq
c_{6}.\label{44}%
\end{equation}
Integrating the differential inequality $(\ref{44})$ from time $t$ to blow-up
time $T^{\ast}$ and using the fact that $\lim\limits_{t\rightarrow T^{\ast}%
}\left\Vert u\left(  .,t\right)  \right\Vert _{\dot{H}^{1}}^{2}=\infty$,
yields
\begin{equation}
1\leq\cos\left(  \frac{1}{y\left(  t\right)  +1}\right)  +c_{6}\left(
T^{\ast}-t\right)  .\label{45}%
\end{equation}
Using $(\ref{10a})$ in $(\ref{45})$, we obtain
\begin{equation}
1-\cos\left(  \frac{1}{y\left(  t\right)  +1}\right)  \leq c_{6}m^{n}\left(
T^{\ast}-t\right)  ^{n+1}\ \text{with }m\in%
%TCIMACRO{\U{2115} }%
%BeginExpansion
\mathbb{N}
%EndExpansion
^{\bigstar},\label{46}%
\end{equation}
for all $t\leq T_{\ast}\ $(see proof of Theorem 2.2). This gives%
\begin{equation}
\sin\left(  \frac{1}{2y\left(  t\right)  +2}\right)  ^{2}\leq c_{6}%
m^{n}\left(  T^{\ast}-t\right)  ^{n+1}.\label{47}%
\end{equation}
Multiplying$\ (\ref{47})$ by $\left\Vert u\left(  .,t\right)  \right\Vert
_{\dot{H}^{1}}^{4}$, we get
\begin{equation}
\left\Vert u\left(  .,t\right)  \right\Vert _{\dot{H}^{1}}^{4}\sin\left(
\frac{1}{\left\Vert u\left(  .,t\right)  \right\Vert _{\dot{H}^{1}}^{2}%
+1}\right)  ^{2}\leq c_{6}m^{n}\left(  T^{\ast}-t\right)  ^{n+1}\left\Vert
u\left(  .,t\right)  \right\Vert _{\dot{H}^{1}}^{4}.\label{48}%
\end{equation}
Thus, there exist a positive constant $\beta_{3}=\min\limits_{t\geq
0}\left\Vert u\left(  .,t\right)  \right\Vert _{\dot{H}^{1}}^{4}$ such that
$\left\Vert u\left(  .,t\right)  \right\Vert _{\dot{H}^{1}}^{4}\geq\beta_{3}$
for all $t\geq0$, this gives
\begin{equation}
\alpha_{3}=\left(  \beta_{3}\sin\left(  \frac{1}{\beta_{3}+1}\right)  \right)
^{2}\leq\left\Vert u\left(  .,t\right)  \right\Vert _{\dot{H}^{1}}^{4}%
\sin\left(  \frac{1}{\left\Vert u\left(  .,t\right)  \right\Vert _{\dot{H}%
^{1}}^{2}+1}\right)  ^{2}.\label{49}%
\end{equation}
Using this estimate in $(\ref{49})$ yields the bound
\begin{equation}
\alpha_{3}\leq c_{6}m^{n}\left(  T^{\ast}-t\right)  ^{n+1}\left\Vert u\left(
.,t\right)  \right\Vert _{\dot{H}^{1}}^{4}.\label{50}%
\end{equation}
This completes the proof of Theorem 2.4.
\end{proof}

In the classical method there is relation between $s$ the order of the Sobolev
spaces $H^{s}$ and the lower bound $\left(  T^{\ast}-t\right)  ^{-\varphi
\left(  s\right)  }$ ,$\varphi$ is a nonlinear function .In our previous
estimates, the bounds on the blow up are independent of $s$. Since $n$ $\in%
%TCIMACRO{\U{2115} }%
%BeginExpansion
\mathbb{N}
%EndExpansion
^{\bigstar}$ we can recover more case, we can get correct rate of blow up in
the form $\left(  T^{\ast}-t\right)  ^{-1}$ in several spaces, in $\dot
{H}^{3/2}$ for $n=$ $3$ and also in $\dot{H}^{1}$ for $n=3$. Theorem 2.2
includes the optimal lower bound for blow-up rate in $\dot{H}^{5/2}$. This
particular case was not achieved \cite{5}, for $n=1$ in $(\ref{11})$ we get a
positive answer to this question
\begin{equation}
\left\Vert u\left(  .,t\right)  \right\Vert _{\dot{H}^{5/2}}\geq\frac{\eta
_{1}}{\left(  T^{\ast}-t\right)  }\text{for }t\leq T_{\ast}.\label{50a}%
\end{equation}
Setting $n=1$ in $(\ref{23})$, Theorem 2.3 gives an improvement of the rate of
blow up $(\ref{6})$ in $\dot{H}^{3/2}\left(  Q\right)  $ obtained in \cite{5}%
\begin{equation}
\left\Vert u\left(  .,t\right)  \right\Vert _{\dot{H}^{3/2}}\geq\frac{\eta
_{2}}{\left(  T^{\ast}-t\right)  ^{\frac{1}{2}}}\text{ for }t\leq T_{\ast
}.\label{51a}%
\end{equation}
Since the results above are valid for all real $m\geq1$, we can control the
distance between $T_{\ast}$ and $T^{\ast}$, where $T^{\ast}-T_{\ast}=\frac
{1}{m}$. This distance can be minimized by choosing a large value of $m$,
which enhances the study of the behavior of strong solutions in the
neighborhood of the blow up.

\end{document}